\documentclass[graybox]{svmult}

\usepackage{amsmath}
\usepackage{amssymb}
\usepackage{mathptmx}       % selects Times Roman as basic font
\usepackage{helvet}         % selects Helvetica as sans-serif font
\usepackage{courier}        % selects Courier as typewriter font
\usepackage{type1cm}        % activate if the above 3 fonts are
\usepackage{makeidx}         % allows index generation
\usepackage{graphicx}        % standard LaTeX graphics tool
\usepackage{multicol}        % used for the two-column index
\usepackage[bottom]{footmisc}% places footnotes at page bottom

\usepackage{color}
\usepackage{url}

\def\res{{\operatorname{res}}}

\def\cR{{\mathcal{R}}}
\def\cO{{\mathcal{O}}}

\newcommand{\bN} { {\mathbb{N}}}
\newcommand{\bC} { {\mathbb{C}}}
\newcommand{\bQ} { {\mathbb{Q}}}
\newcommand{\bZ} { {\mathbb{Z}}}

\newcommand{\cP} { {\mathcal{P}}}
\newcommand{\cE} { {\mathcal{E}}}

\newcommand{\si} { {\sigma}}

\newcommand{\EOP} { {\hfill $\Box$}}

\makeindex             % used for the subject index
\begin{document}

\title*{How to generate all possible rational Wilf-Zeilberger pairs?\\
\bigskip
{\small \text{Dedicated to the memory of Jonathan M.\ Borwein and Ann Johnson}}}
\titlerunning{How to generate all possible rational Wilf--Zeilberger pairs?}

\author{Shaoshi Chen}

\institute{Shaoshi Chen \at KLMM, Academy of Mathematics and Systems Science, Chinese Academy of Sciences, \\ Beijing 100190, (China)\\
School of Mathematical Sciences, University of Chinese Academy of Sciences, \\ Beijing 100049, (China)\\
\email{schen@amss.ac.cn} \smallskip \\
This work was supported by the NSFC grants 11501552, 11688101 and
by the Frontier Key Project (QYZDJ-SSW-SYS022) and the Fund of the Youth Innovation Promotion Association, CAS}
\maketitle

%\abstract*{Each chapter should be preceded by an abstract (10--15 lines long) that summarizes the content. The abstract will appear \textit{online} at \url{www.SpringerLink.com} and be available with unrestricted access. This allows unregistered users to read the abstract as a teaser for the complete chapter. As a general rule the abstracts will not appear in the printed version of your book unless it is the style of your particular book or that of the series to which your book belongs.
%Please use the 'starred' version of the new Springer \texttt{abstract} command for typesetting the text of the online abstracts (cf. source file of this chapter template \texttt{abstract}) and include them with the source files of your manuscript. Use the plain \texttt{abstract} command if the abstract is also to appear in the printed version of the book.}

\abstract{ A Wilf--Zeilberger pair $(F, G)$ in the discrete case satisfies the equation
\[ F(n+1, k) - F(n, k) = G(n, k+1) - G(n, k).\]
We present a structural description of all possible rational Wilf--Zeilberger pairs and their continuous and mixed analogues.
}

\section{Introduction} \label{SECT:intro}

The Wilf--Zeilberger (abbr.\ WZ) theory~\cite{Wilf1992, WilfZeilberger1992, PWZbook1996} has become a bridge between symbolic computation and combinatorics.
Through this bridge, not only classical combinatorial identities from handbooks and long-standing conjectures in combinatorics, such as Gessel's conjecture~\cite{KKZ2009, Bostan2010} and $q$-TSPP conjecture~\cite{KKZ2011}, are proved algorithmically, but also some new identities and conjectures related to mathematical constants, such as $\pi$ and zeta values,  are discovered via computerized guessing~\cite{Gessel1995, Borwein2004, Sun2011, CHZ2016}.

WZ-pair is one of leading concepts in the WZ theory that was originally introduced in~\cite{WilfZeilberger1992} with a recent brief description in~\cite{Tefera2010}.
In the discrete case,  a WZ-pair $(F(n, k), G(n, k))$ satisfies the WZ equation
\[ F(n+1, k) - F(n, k) = G(n, k+1) - G(n, k),
\]
where both $F$ and $G$ are hypergeometric terms, i.e., their shift quotients with respect to $n$ and $k$ are rational
functions in $n$ and $k$, respectively.
Once a WZ-pair is given, one can sum on both sides of the above equation over $k$ from $0$ to $\infty$ to get
\[\sum_{k=0}^{\infty} F(n+1, k) - \sum_{k=0}^\infty F(n, k)= \lim_{k\rightarrow\infty} G(n, k+1) - G(n, 0).\]
If $G(n, 0)$ and $\lim_{k\rightarrow\infty} G(n, k+1)$ are $0$ then we obtain
\[\sum_{k=0}^{\infty} F(n+1, k) = \sum_{k=0}^\infty F(n, k),\]
which implies that $\sum_{k=0}^\infty F(n, k)$ is independent of $n$. Thus, we get
the identity  $\sum_{k=0}^\infty F(n, k) = c$, where the constant $c$ can be determined by evaluating the sum for one value of $n$.
We may also get a companion identity by summing the WZ-equation over $n$.
For instance, the pair $(F, G)$ with
\[F = \frac{\binom{n}{k}^2}{\binom{2n}{n}} \quad \text{and} \quad  G= {\frac {( 2k-3n-3) {k}^{2}}{ 2(2n+1)( -n-1+k)^{2}}} \cdot \frac{\binom{n}{k}^2}{\binom{2n}{n}}
 \]
leads to two identities
\[\sum_{k=0}^\infty \binom{n}{k}^2 = \binom{2n}{n} \quad \text{and} \quad \sum_{n=0}^\infty \frac{(3n-2k+1)}{2(2n+1)\binom{2n}{n}} \binom{n}{k}^2 = 1. \]

Besides to prove combinatorial identities, WZ-pairs have many other applications.
One of the applications can be traced back to Andrei Markov's 1890 method for convergence-acceleration of series for
computing $\zeta(3)$, which leads to the Markov-WZ method~\cite{MZ2004, Kondratieva2005, Mohammed2005}.
WZ-pairs also play a central role in the study of finding Ramanujan-type and Zeilberger-type series for constants
involving $\pi$ in~\cite{EZ1994, Guillera2002, Guillera2006, Guillera2010, Guillera2013, Liu2012, Zudilin2011, HKS2018}, zeta values~\cite{Pilehrood2008a, Pilehrood2008b}
and their $q$-analogues~\cite{Pilehrood2011, GuoLiu2018, GuoZudilin2018}. Most recent applications are related to
congruences and super congruences~\cite{Zudilin2009, Long2011, Sun2011, Sun2012, Sun2013, Sun2013b, Guo2017, Guo2018}.

For appreciation we select some remarkable $(q)$-series about $\pi, \zeta(3)$ together with (super)-congruences whose proofs can be obtained via WZ-pairs
as follows (this list is surely not comprehensive):

\begin{enumerate}
\item Ramanujan's series for $1/\pi$:  first recorded in Ramanujan's second notebook, proved by Bauer in~\cite{Bauer1859}, and by Ekhad and Zeilberger using WZ-pairs in~\cite{EZ1994}.
For a nice survey on Ramanujan's series, see~\cite{BBC2009}.
\[\frac{2}{\pi} = \sum_{k=0}^\infty \frac{4k+1}{(-64)^k} \binom{2k}{k}^3 .
\]

\item Guillera's series for $1/\pi^2$: found and proved by Guillera in 2002 using WZ-pairs~\cite{Guillera2002}. For more results on Ramanujan-type series for~$1/\pi^2$, see Zudilin's surveys~\cite{Zudilin2007, Zudilin2011}.
  \[\frac{128}{\pi^2} = \sum_{k=0}^\infty (-1)^k \binom{2k}{k}^5 \frac{820k^2 + 180k + 13}{2^{20k}} .\]

  \item Guillera's Zeilberger-type series for $\pi^2$: found and proved by Guillera using WZ-pairs in~\cite{Guillera2008}.

  \[\frac{\pi^2}{2} = \sum_{k=1}^\infty \frac{(3k-1)16^k}{k^3\binom{2k}{k}^3}.\]

  \item Markov--Ap\'ery's series for $\zeta(3)$: first discovered by Andrei Markov in 1890, used by Ap\'ery for his irrationality proof, and proved by Zeilberger using WZ-pairs in~\cite{Zeilberger1993}.
\[
\zeta(3)  = \frac{5}{2} \sum_{k=1}^\infty \frac{(-1)^{k-1}}{k^3\binom{2k}{k}}.
\]

\item Amdeberhan's series for $\zeta(3)$:  proved by Amdeberhan in 1996 using WZ-pairs~\cite{Amdeberhan1996}.
\[\zeta(3) = \frac{1}{4} \sum_{k=1}^\infty (-1)^{k-1}\frac{56k^2 - 32k+5}{k^3 (2k-1)^2 \binom{2k}{k}\binom{3k}{k}}.\]

  \item Bailey--Borwein--Bradley identity: experimentally discovered and proved by Bailey et al. in~\cite{BBB2006}, a proof using the Markov-WZ method is given in~\cite{Pilehrood2008b, Pilehrood2008b} and its $q$-analogue
  is presented in~\cite{Pilehrood2011}.
   \[  \sum_{k=0}^\infty \zeta(2k+2) z^{2k} = 3 \sum_{k=1}^\infty \frac{1}{\binom{2k}{k}(k^2-z^2)} \prod_{m=1}^{k-1} \frac{m^2-4z^2}{m^2-z^2}, \quad \text{$z\in \bC$ with $|z|<1$}. \]

   \item van Hamme's supercongruence I: first conjectured by van Hamme~\cite{Hamme1997}, proved by Mortenson~\cite{Mortenson2008} using $_6F_5$ transformations and by Zudilin~\cite{Zudilin2009}
  using WZ-pairs.
  \[\sum_{k=0}^{\frac{p-1}{2}} \frac{4k+1}{(-64)^k} \binom{2k}{k}^3 \equiv p(-1)^{\frac{p-1}{2}} (\text{mod}~p^3), \]
where $p$ is an odd prime and the multiplicative inverse of $(-64)^k$ should be computed modulo $p^3$.

  \item van Hamme's supercongruence II: first conjectured by van Hamme~\cite{Hamme1997}, proved by Long~\cite{Long2011} using hypergeometric evaluation identities, one of which is obtained by Gessel using WZ-pairs in~\cite{Gessel1995}.
  \[\sum_{k=0}^{\frac{p-1}{2}}{\frac{6k+1}{256^k}}\binom{2k}{k}^3 \equiv p(-1)^{\frac{p-1}{2}} (\text{mod}~p^4),\]
where $p>3$ is a prime and the multiplicative inverse of $(256)^k$ should be computed modulo $p^4$.

  \item Guo's $q$-analogue of van Hamme's supercongruence I: discovered and proved recently by Guo using WZ-pairs in~\cite{Guo2018}.
  \[\sum_{k=0}^{\frac{p-1}{2}}(-1)^k q^{k^2}[4k+1]_q\frac{(q;q^2)_k^3}{(q^2;q^2)_k^3}
\equiv [p]_qq^{\frac{(p-1)^2}{4}} (-1)^{\frac{p-1}{2}}\pmod{[p]_q^3},\]
where for $n\in \bN$, $(a;q)_n := (1-a)(1-aq)\cdots (1-aq^{n-1})$ with $(a;q)_0 = 1$, $[n]_q= 1+q+\cdots + q^{n-1}$  and $p$ is an odd prime.
  \item Hou--Krattenthaler--Sun's $q$-analogue of Guillera's Zeilberger-type series for $\pi^2$:  inspired by a recent conjecture on supercongruence  by Guo in~\cite{Guo2018b}, and proved using WZ-pairs in~\cite{HKS2018}.
  This work is also connected to other emerging developments on $q$-analogues of series for famous constants and formulae~\cite{Sun2018, GuoZudilin2018, GuoLiu2018}.
  \[2\sum_{k=0}^\infty q^{2k^2+2k}(1+q^{2k^2+2}-2q^{4k+3}) \frac{(q^2;q^2)_k^3}{(q;q^2)^3_{k+1}(-1; q)_{2k+3}} = \sum_{k=0}^\infty \frac{q^{2k}}{(1-q^{2k+1})^2} .  \]

\end{enumerate}

%
%\begin{equation}
%
%\end{equation}
%
%\begin{equation}
%
%\end{equation}
%
%\begin{equation}
%
%\end{equation}
%

For applications, it is crucial to have WZ-pairs at hand. In the previous work, WZ-pairs are obtained
either by guessing from the identities to be proved using Gosper's algorithm or by certain transformations from a given WZ-pair~\cite{Gessel1995}.
Riordan in the preface of his book~\cite{Riordan1968} commented that~\lq\lq the central fact developed is that identities are both
inexhaustible and unpredictable; the age-old dream of putting order in this chaos is doomed to failure~\rq\rq. As an optimistic respond to Riordan's comment,
Gessel in his talk\footnote{The talk was given at the Waterloo Workshop in Computer Algebra (in honor of Herbert Wilf's 80th birthday), Wilfrid Laurier University, May 28, 2011.
For the talk slides, see the link:~\url{http://people.brandeis.edu/~gessel/homepage/slides/wilf80-slides.pdf}} on the WZ method motivated with some examples that~\lq\lq WZ forms bring order to this chaos~\rq\rq, where WZ-forms are a multivariate generalization of WZ-pairs~\cite{Zeilberger1993}. With the hope of discovering more combinatorial identities in an intrinsic and algorithmic way, it is natural and challenging to ask the following question.

\begin{problem}\label{PB:WZ}
How to generate all possible WZ-pairs algorithmically?
\end{problem}

This problem seems quite open, but every promising project needs a starting point.
In~\cite{Liu2015}, Liu had described the structure of
a special class of analytic WZ-functions with $F=G$ in terms of Rogers--Szeg\"{o} polynomials and Stieltjes--Wigert polynomials in the $q$-shift case.
In~\cite{Sun2012}, Sun studied the relation between generating functions of $F(n, k)$ and $G(n, k)$ if $(F, G)$ is a WZ-pair and applied this relation
to prove some combinatorial identities.
In this paper, we solve the problem completely for the first non-trivial case, namely, the case of rational WZ-pairs.
To this end, let us first introduce some notations. Throughout this paper, let $K$ be a field of characteristic zero and $K(x, y)$ be the field of rational functions
in $x$ and $y$ over $K$. Let $D_x = \partial/\partial_x$ and $D_y = \partial/\partial_y$ be the usual
derivations with respect to $x$ and $y$, respectively. The shift operators $\si_x$ and $\si_y$ are defined respectively as
\[\si_x(f(x, y)) = f(x+1, y) \quad \text{and} \quad \si_y(f(x, y)) = f(x, y+1) \quad \text{for $f\in K(x, y)$.}\]
For any $q\in K\setminus \{0\}$, we define the $q$-shift operators $\tau_{q, x}$ and $\tau_{q, y}$ respectively as
\[\tau_{q, x}(f(x, y)) = f(qx, y) \quad \text{and} \quad \tau_{q, y}(f(x, y)) = f(x, qy) \quad \text{for $f\in K(x, y)$.}\]
For $z\in \{x, y\}$, let~$\Delta_z$ and $\Delta_{q, z}$ denote the difference and $q$-difference operators defined by $\Delta_z(f) = \si_z(f) - f$ and $\Delta_{q, z}(f)= \tau_{q, z}(f)-f$ for $f\in K(x, y)$, respectively.

\begin{definition}
Let $\partial_x \in \{D_x, \Delta_x, \Delta_{q, x}\}$ and $\partial_y \in \{D_y, \Delta_y, \Delta_{q, y}\}$. A pair $(f, g)$ with $f, g\in K(x, y)$
is called a \emph{WZ-pair} with respect to $(\partial_x, \partial_y)$ in $K(x, y)$ if $\partial_x(f) = \partial_y(g)$.
\end{definition}
The set of all rational WZ-pairs in $K(x, y)$ with respect to $(\partial_x, \partial_y)$ forms a linear space over $K$, denoted by $\cP_{(\partial_x, \partial_y)}$.
A WZ-pair $(f, g)$ with respect to $(\partial_x, \partial_y)$ is said to be \emph{exact\footnote{This is motivated by the fact that a differential form $\omega = gdx + fdy$ with $f, g \in K(x, y)$ is exact in $K(x, y)$ if and only if $f = D_y(h)$ and $g = D_x(h)$ for some $h\in K(x, y)$.}} if there exists $h\in K(x, y)$ such that $f= \partial_y(h)$
and $g = \partial_x(h)$.
Let $\cE_{(\partial_x, \partial_y)}$ denote the set of all exact WZ-pairs with respect to $(\partial_x, \partial_y)$, which
forms a subspace of $\cP_{(\partial_x, \partial_y)}$.
The goal of this paper is to provide an explicit description of the structure of  the quotient space $\cP_{(\partial_x, \partial_y)}/\cE_{(\partial_x, \partial_y)}$.

The remainder of this paper is organized as follows.
As our key tools, residue criteria for rational integrability and summability
are recalled in Section~\ref{SECT:res}. In Section~\ref{SECT:structure}, we present structure theorems
for rational WZ-pairs in three different settings. This paper
ends with a conclusion along with some remarks on the future research.

\section{Residue criteria}  \label{SECT:res}

In this section, we recall the notion of residues and their ($q$-)discrete analogues for rational functions and
some residue criteria for rational integrability and summability from~\cite{BronsteinBook, ChenSinger2012, HouWang2015}.

Let $F$ be a field of characteristic zero and $F(z)$ be the field of rational functions
in $z$ over $F$. Let~$D_z$ be the usual derivation on $F(z)$ such that $D_z(z)=1$ and $D_z(c)$=0 for all $c\in F$.
A rational function $f\in F(z)$ is said to be \emph{$D_z$-integrable} in $F(z)$ if $f = D_z(g)$ for some $g\in F(z)$.
By the irreducible partial fraction decomposition, one can always uniquely write $f\in F(z)$ as
\begin{equation}\label{EQ:pfdc}
f = q + \sum_{i=1}^n \sum_{j=1}^{m_i} \frac{a_{i, j}}{ d_i^j},
\end{equation}
where~$q, a_{i, j}, d_i \in F[z]$, $\deg_z(a_{i, j})< \deg_z(d_i)$ and the $d_i$'s
are distinct irreducible and monic polynomials.
We call $a_{i, 1}$ the \emph{pseudo $D_z$-residue} of $f$ at $d_i$,
denoted by $\text{pres}_{D_z}(f, d_i)$. For an irreducible polynomial $p\in F[z]$, we let $\cO_p $ denote the set
\[\cO_p := \left\{\frac{a}{b}\in F(z)\mid \text{$a, b \in F[z]$ with~$\gcd(a, b)$ and ${p}\nmid b$}\right\},\]
and let $\cR_p$ denote the set $\{f\in F(z)\mid pf\in \cO_p\}$. If $f\in \cR_p$, the pseudo-residue $\text{pres}_{D_z}(f, p)$ is called the \emph{$D_z$-residue}
of $f$ at $p$, denoted by $\res_{D_z}(f, p)$. The following example shows that pseudo-residues may not be
the obstructions for $D_z$-integrability in $F(z)$.

\begin{example}
Let $F := \bQ$ and $f = (1-x^2)/(x^2+1)^2$. Then the irreducible partial fraction decomposition of $f$
is of the form
\[f  = \frac{2}{(x^2+1)^2}- \frac{1}{x^2+1}.\]
The pseudo-residue of $f$ at $x^2+1$ is $-1$, which is nonzero. However, $f$ is $D_z$-integrable in $F(z)$ since $f = D_z(x/(x^2+1))$.
\end{example}
The following lemma shows that $D_z$-residues are the only obstructions for $D_z$-integrability of rational functions
with squarefree denominators, so are pseudo-residues if $F$ is algebraically closed.

\begin{lemma}\cite[Proposition 2.2]{ChenSinger2012}\label{LM:dres}
Let~$f=a/b\in F(z)$ be such that $a, b\in F[z]$, $\gcd(a, b)=1$.  If $b$ is squarefree, then $f$ is $D_z$-integrable
in $F(z)$ if and only if $\res_{D_z}(f, d)=0$ for any irreducible factor $d$ of $b$. If $F$ is algebraically closed, then
$f$ is $D_z$-integrable in $F(z)$ if and only if
$\text{pres}_{D_z}(f, z-\alpha)=0$ for any root $\alpha$ of the denominator $b$.
\end{lemma}

By the Ostrogradsky--Hermite reduction~\cite{Ostrogradsky1845, Hermite1872, BronsteinBook}, we can decompose a rational function $f\in F(z)$
as  $f = D_z(g) + a/b$, where $g\in F(z)$ and~$a, b \in F[z]$ are such that $\deg_z(a)<\deg_z(b), \gcd(a, b)=1$,
and $b$ is a squarefree polynomial in $F[z]$. By Lemma~\ref{LM:dres}, $f$ is $D_z$-integrable in $F(z)$ if and only if $a=0$.

%%Assume that $h = \sum_{i=1}^n \frac{a_i}{b_i}$ with $a_i, b_i \in F[z]$ satisfying $\deg_z(a_i)<\deg_z(b_i)$ and
%%$b_i$'s are distinct irreducible polynomials in $F[z]$.
%%As an immediate corollary of Lemma~\ref{LM:dres}, we have $a_i = 0$ for all $i$ with $1\leq i \leq n$ if $h$
%%is integrable in $F(z)$.

We now recall the ($q$-)discrete analogue of
$D_z$-residues introduced in~\cite{ChenSinger2012, HouWang2015}.
Let $\phi$ be an automorphism of $F(z)$ that fixes $F$. For a polynomial $p\in F[z]$,
we call the set $\{\phi^i(p)\mid i\in \bZ\}$ the \emph{$\phi$-orbit} of $p$, denoted by $[p]_{\phi}$.
Two polynomials $p, q\in F[z]$ are said to be $\phi$-equivalent (denoted as $p\sim_{\phi} q$) if they are in the same $\phi$-orbit, i.e., $p = \phi^i(q)$ for some $i\in \bZ$.
For any $a, b \in F(z)$ and $m\in \bZ$, we have
\begin{equation}\label{EQ:phi}
\frac{a}{\phi^m(b)} = \phi(g) - g + \frac{\phi^{-m}(a)}{b},
\end{equation}
where $g$ is equal to $\sum_{i=0}^{m-1} \frac{\phi^{i-m}(a)}{\phi^i(b)}$ if $m\geq 0$,  and equal to $-\sum_{i=0}^{-m-1}\frac{\phi^i(a)}{\phi^{m+i}(b)}$ if $m<0$.

Let $\si_z$ be the shift operator with respect to $z$ defined by $\si_z(f(z)) = f(z+1)$. Note that $\si_z$ is an automorphism of $F(z)$ that fixes $F$.
A rational function $f\in F(z)$ is said to be \emph{$\si_z$-summable} in $F(z)$ if $f = \si_z(g)-g$ for some $g\in F(z)$.
For any $f\in F(z)$, we can uniquely decompose it into the form
\begin{equation}\label{EQ:pfdd}
f = p(z) + \sum_{i=1}^n \sum_{j=1}^{m_i} \sum_{\ell=0}^{e_{i, j}} \frac{a_{i, j, \ell}}{ \si_z^\ell (d_i)^j},
\end{equation}
where $p, a_{i, j, \ell}, d_i\in F[z]$, $\deg_z(a_{i, j, \ell})< \deg_z(d_i)$ and the $d_i$'s
are irreducible and monic polynomials such that no two of them are $\si_z$-equivalent. We call the sum $\sum_{\ell=0}^{e_{i, j}} \si_z^{-\ell}(a_{i, j, \ell})$
the \emph{$\si_z$-residue} of $f$ at $d_i$ of multiplicity $j$, denoted by $\res_{\si_z}(f, d_i, j)$.
Recently, the notion of $\si_z$-residues has been generalized to the case of rational functions over elliptic curves~\cite[Appendix B]{Dreyfus2018}.
The following lemma is a discrete analogue of Lemma~\ref{LM:dres} which
shows that $\si_z$-residues are the only obstructions for $\si_z$-summability in the field~$F(z)$.

\begin{lemma}\cite[Proposition 2.5]{ChenSinger2012}\label{LM:sres}
Let~$f=a/b\in F(z)$ be such that $a, b\in F[z]$ and $\gcd(a, b)=1$. Then $f$ is $\si_z$-summable in $F(z)$ if and only if
$\res_{\si_z}(f, d, j)=0$ for any irreducible factor $d$ of the denominator $b$ of any multiplicity $j\in \bN$.
\end{lemma}
By Abramov's reduction~\cite{Abramov1975, Abramov1995b}, we can decompose a rational function $f\in F(z)$ as
\[f = \Delta_z(g) + \sum_{i=1}^n \sum_{j=1}^{m_i}\frac{a_{i, j}}{b_i^j},\]
where $g\in F(z)$ and $a_{i, j}, b_i \in F[z]$ are such that $\deg_z(a_{i, j})<\deg_z(b_i)$ and the
$b_i$'s are irreducible and monic polynomials in distinct $\si_z$-orbits. By Lemma~\ref{LM:sres}, $h$ is $\si_z$-summable in $F(z)$ if and only if
$a_{i, j} = 0$ for all $i, j$ with $1\leq i \leq n$ and $1\leq j \leq m_i$.

Let~$q$ be a nonzero element of $F$ such that $q^m \neq 1$ for all nonzero $m\in \bZ$ and let $\tau_{q, z}$
be the $q$-shift operator with respect to $z$ defined by $\tau_{q, z}(f(z)) = f(qz)$. Since $q$ is nonzero, $\tau_{q, z}$
is an automorphism of $F(z)$ that fixes $F$.
A rational function $f\in F(z)$ is said to be \emph{$\tau_{q, z}$-summable} in $F(z)$ if $f = \tau_{q, z}(g)-g$ for some $g\in F(z)$.
For any $f\in F(z)$, we can uniquely decompose it into the form
\begin{equation}\label{EQ:pfdq}
f = c + zp_1 + \frac{p_2}{z^s} + \sum_{i=1}^n \sum_{j=1}^{m_i} \sum_{\ell=0}^{e_{i, j}} \frac{a_{i, j, \ell}}{ \tau_{q, z}^\ell (d_i)^j},
\end{equation}
where $c\in F, s, n, m_i, e_{i, j} \in \bN$ with $s\neq 0$,
and $p_1, p_2, a_{i, j, \ell}, d_i\in F[z]$ are such that $\deg_z(p_2)<s$, $\deg_z(a_{i, j, \ell})< \deg_z(d_i)$,
and $p_2$ is either zero or has nonzero constant term, i.e., $p_2(0)\neq 0$. Moreover, the $d_i$'s are irreducible and monic polynomials in distinct $\tau_{q, z}$-orbits and $z\nmid d_i$ for all $i$ with $1\leq i \leq n$. We call the constant $c$ the \emph{$\tau_{q, z}$-residue} of $f$ at infinity, denoted by $\res_{\tau_{q, z}}(f, \infty)$ and call   the sum $\sum_{\ell=0}^{e_{i, j}} \tau_{q, z}^{-\ell}(a_{i, j, \ell})$
the \emph{$\tau_{q, z}$-residue} of $f$ at $d_i$ of multiplicity $j$, denoted by $\res_{\tau_{q, z}}(f, d_i, j)$.
A $q$-analogue of Lemma~\ref{LM:sres} is as follows.

\begin{lemma}\cite[Proposition 2.10]{ChenSinger2012}\label{LM:qres}
Let~$f=a/b\in F(z)$ be such that $a, b\in F[z]$ and $\gcd(a, b)=1$. Then $f$ is $\tau_{q, z}$-summable in $F(z)$ if and only if $\res_{\tau_{q, z}}(f, \infty) =0$ and
$\res_{\tau_{q, z}}(f, d, j)=0$ for any irreducible factor $d$ of the denominator $b$ of any multiplicity $j\in \bN$.
\end{lemma}

By a $q$-analogue of Abramov's reduction~\cite{Abramov1995b}, we can decompose a rational function $f\in F(z)$ as
\[f = \Delta_{q, z}(g) + c +\sum_{i=1}^n \sum_{j=1}^{m_i}\frac{a_{i, j}}{b_i^j},\]
where~$g\in F(z), c\in F,$ and~$a_{i, j}, b_i \in F[z]$ are such
that $\deg_z(a_{i, j})<\deg_z(b_i)$ and
the $b_i$'s are irreducible and monic polynomials in distinct $\si_z$-orbits and $\gcd(z, b_i)=1$ for all $i$ with $1\leq i \leq n$. By Lemma~\ref{LM:qres}, $f$ is $\tau_{q,z}$-summable in $F(z)$ if and only if
$c=0$ and~$a_{i, j} = 0$ for all $i, j$ with $1\leq i \leq n$ and $1\leq j \leq m_i$.
\begin{remark}
Note that pseudo-residues are essentially different from residues in the
differential case, but not needed in the shift and $q$-shift cases.
\end{remark}

\section{Structure theorems}\label{SECT:structure}

%The set of all WZ-pairs with respect to $(\partial_x, \partial_y)$ forms a linear space over $K$, denoted by $\cP_{(\partial_x, \partial_y)}$.
%A WZ-pair $(f, g)$ with respect to $(\partial_x, \partial_y)$ is said to be \emph{exact\footnote{This is motivated by the fact that a differential form $\omega = gdx + fdy$ with $f, g \in K(x, y)$ is exact in $K(x, y)$ if and only if $f = D_y(h)$ and $g = D_x(h)$ for some $h\in K(x, y)$.}} if there exits $h\in K(x, y)$ such that $f= \partial_y(h)$
%and $g = \partial_y(h)$.
%Let $\cE_{(\partial_x, \partial_y)}$ denote the set of all exact WZ-pairs with respect to $(\partial_x, \partial_y)$, which
%forms a subspace of $\cP_{(\partial_x, \partial_y)}$.
In this section, we present structure theorems for rational WZ-pairs in terms of some special pairs.
Throughout this section, we will assume that $K$ is an algebraically closed field of characteristic zero and let $\partial_x \in \{D_x, \Delta_x, \Delta_{q, x}\}$ and $\partial_y \in \{D_y, \Delta_y, \Delta_{q, y}\}$.

We first consider the special case that $q\in K$ is a root of unity.
Assume that $m$ is the minimal positive integer such that~$q^m=1$. For any $f\in K(x, y)$,
it is easy to show that~$\tau_{q, y}(f) = f$ if and only if~$f\in K(x)(y^m)$.
Note that $K(x, y)$ is a finite algebraic extension of $K(x)(y^m)$ of degree $m$.
In the following theorem, we show that WZ-pairs in this special case
are of a very simple form.

\begin{theorem}\label{THM:specialq}
Let~$\partial_x \in \{D_x, \Delta_x, \Delta_{q, x}\} $ and~$f, g\in K(x, y)$ be such that $\partial_x(f) = \Delta_{q, y}(g)$.
Then there exist rational functions $h\in K(x, y)$ and $a, b\in K(x, y^m)$ such that $\partial_x(a)=0$ and
\[f  = \Delta_{q, y}(h) + a  \quad \text{and} \quad g  = \partial_x(h) + b.\]
Moreover, we have $a \in K(y^m)$ if $\partial_x\in \{D_x, \Delta_x\}$ and $a\in K(x^m, y^m)$ if $\partial_x = \Delta_{q, x}$.
\end{theorem}
\begin{proof}
By Lemma 2.4 in~\cite{ChenSinger2014}, any rational function~$f\in K(x, y)$ can be decomposed as
\begin{equation}\label{rook}
f = \Delta_{q, y}(h)+ a, \quad \text{where~$h\in K(x, y)$ and~$a\in K(x)(y^m)$}.
\end{equation}
Moreover, $f$ is~$\tau_{q, y}$-summable in~$K(x, y)$ if and only if~$a=0$. Then
\[\partial_x(f) = \Delta_{q, y}(\partial_x(h))+ \partial_x(a).\]
Note that $\partial_x(a) \in K(x)(y^m)$, which implies that $\partial_x(a)=0$ because $\partial_x(f)$ is~$\tau_{q, y}$-summable in~$K(x, y)$.
Then $\Delta_{q, y}(g) = \Delta_{q, y}(\partial_x(h))$. So $g = \partial_x(h) + b$ for some $b\in K(x, y^m)$. This completes the proof.
\EOP
\end{proof}

From now on, we assume that $q$ is not a root of unity. We will investigate WZ-pairs in three different cases according to the choice of the pair $(\partial_x, \partial_y)$.

%%\begin{theorem}\label{THM:qru}
%%Let~$q$ be such that~$q^m=1$ with~$m$ minimal and let $f \in k(x, y)$. Assume that $f = \tau_{x, q}(g) - g + c$ with
%%$g\in k(x, y)$ and~$c\in k(y)(x^m)$ is a {$\tau_{x, q}$-reduced form} of $f$. Then $f$ is exact with respect to~$(\tau_{x, q}, \partial_y)$
%%with $\partial_y\in \{\Delta_y, D_y\}$ if and only if $c = \partial_y(d)$ for some $d\in k(y)(x^m)$.
%%\end{theorem}
%%\begin{proof}
%%The sufficiency is clear. To show the necessity, we assume that $f$ is exact with respect to~$(\tau_{x, q}, \partial_y)$
%%with $\partial_y\in \{\Delta_y, D_y\}$, so is $c$, i.e.,
%%$c = \Delta_{x, q}(u) + \partial_y(v)$ for some $u, v \in k(x, y)$.
%%Write $u = \sum_{i=0}^{m-1} u_i x^i$ and $v = \sum_{i=0}^{m-1}v_i x^i$ with $u_i, v_i \in k(y, x^m)$.
%%Then we have
%%\[c = u_1 (q-1)x + \cdots + u_{m-1} (q^{m-1}-1) x^{m-1} + \sum_{i=0}^{m-1} \partial_y(v_i) x^i. \]
%%Since $1, x, \ldots, x^{m-1}$ are linearly independent in $k(x, y)$ over $k(y, x^m)$, we get that
%%$c = \partial_y(v_0)$. \qed
%%\end{proof}
%%
%%

\subsection{The differential case} \label{SSECT:cont}

In the continuous setting, we consider WZ-pairs with respect to $(D_x, D_y)$, i.e., the
pairs of the form $(f, g)$ with $f, g\in K(x, y)$ satisfying $D_x(f)= D_y(g)$.
\begin{definition}\label{DEF:logd}
A WZ-pair $(f, g)$ with respect to $(D_x, D_y)$ is called a \emph{log-derivative} pair if there exists nonzero $h\in K(x, y)$ such that $f = D_y(h)/h$
and $g = D_x(h)/h$.
\end{definition}
The following theorem shows that any WZ-pair in the continuous case is a linear combination of exact and log-derivative pairs,
which was first proved by Christopher in~\cite{Christopher1999} and then extended to the multivariate case in~\cite{Zoladek1998, ChenThesis2011}.

\begin{theorem}\label{THM:cstructure}
Let~$f, g\in K(x, y)$ be such that $D_x(f) = D_y(g)$. Then there exist rational functions $a, b_1, \ldots, b_n \in K(x, y)$ and nonzero
constants~$c_1, \ldots, c_n\in K$ such that
\[f  = D_y(a) + \sum_{i =1}^n c_i \frac{D_y(b_i)}{b_i} \quad \text{and} \quad g  = D_x(a) + \sum_{i =1}^n c_i \frac{D_x(b_i)}{b_i}.\]
\end{theorem}
\begin{proof}
The proof in the case when $K$ is the field of complex numbers can be found in~\cite[Theorem 2]{Christopher1999} and in the case when $K$ is any algebraically closed
field of characteristic zero can be found in~\cite[Theorem 4.4.3]{ChenThesis2011}.
\end{proof}

\begin{corollary}
The quotient space $\cP_{(D_x, D_y)}/\cE_{(D_x, D_y)}$ is spanned over $K$ by the set
\[ \{(f, g) + \cE_{(D_x, D_y)} \mid \text{$f, g \in K(x, y)$ such that $(f, g)$ is a log-derivative pair} \}.\]
\end{corollary}

\begin{remark}
A differentiable function $h(x, y)$ is said to be hyperexponential over $\bC(x, y)$ if $D_x(h)= f h$ and $D_y(h)= gh$  for some $f, g\in \bC(x, y)$.
The above theorem enables us to obtain the multiplicative structure of hyperexponential functions, i.e., any hyperexponential function $h(x, y)$
can be written as
$h= \exp(a) \cdot \prod_{i=1}^n b_i^{c_i}$ for some $a, b_i\in \bC(x, y)$ and $c_i\in \bC$.
\end{remark}

\subsection{The ($q$)-shift case} \label{SSECT:dq}
In the discrete setting,  we consider WZ-pairs with respect to $(\partial_x, \partial_y)$ with $\partial_x \in  \{\Delta_x, \Delta_{q, x}\}$
and $\partial_y \in \{\Delta_y, \Delta_{q, y}\}$,  i.e., the
pairs of the form $(f, g)$ with $f, g\in K(x, y)$ satisfying $\partial_x(f)= \partial_y(g)$.

Let $\theta_x\in \{\si_x, \tau_{q, x}\}$ and $\theta_y\in \{\si_y, \tau_{q, y}\}$. For any nonzero $m\in \bZ$, $\theta_x^m$ is also an automorphism on $K(x, y)$ that fixes $K(y)$, i.e., for any $f\in K(x, y)$, $\theta_x^m(f) = f$ if and only if $f\in K(y)$.
The ring of polynomials in $\theta_x$ and $\theta_y$ over $K$ is denoted by
$K[\theta_x, \theta_y]$. For any $p = \sum_{i, j} c_{i, j}\theta_x^i\theta_y^j \in K[\theta_x, \theta_y]$ and $f\in K(x, y)$, we define the action~$p \bullet f = \sum_{i, j} c_{i, j} \theta_x^i(\theta_y^j(f))$. Then $K(x, y)$ can be viewed as a $K[\theta_x,  \theta_y]$-module. Let~$G = \langle \theta_x, \theta_y\rangle$ be the free abelian group generated by~$\theta_x$ and~$\theta_y$.
Let~$f\in K(x, y)$ and $H$ be a subgroup of~$G$. We call the set $\{c\theta(f)\mid c\in K\setminus\{0\}, \theta\in H\}$ the \emph{$H$-orbit} at~$f$, denoted by $[f]_H$.
Two elements~$f, g\in K(x, y)$ are said to be $H$-equivalent if~$[f]_H = [g]_H$, denoted by $f \sim_H g$. The relation
$\sim_H$ is an equivalence relation. A rational function $f\in K(x, y)$ is said to be \emph{$(\theta_x, \theta_y)$-invariant} if there exist $m, n\in \bZ$, not all zero, such that $\theta_x^m\theta_y^n(f) = f$.
All possible $(\theta_x, \theta_y)$-invariant rational functions have been completely characterized in~\cite{AbramovPetkovsek2002a, Ore1930, Sato1990, CFFJ2012, CCFFL2015}. We summarize the characterization as follows.

\begin{proposition}\label{THM:intlinear}
Let~$f\in K(x, y)$ be $(\theta_x, \theta_y)$-invariant, i.e., there exist $m, n\in \bZ$, not all zero, such that
$\theta_x^m\theta_y^n(f) = f$.  Set $\bar n = n/\gcd(m, n)$ and $\bar m = m/ \gcd(m ,n)$. Then
\begin{itemize}
\item[1.] if $\theta_x = \si_x$ and $\theta_y = \si_y$, then $f = g(\bar nx-\bar my)$ for some $g\in K(z)$;
\item[2.] if $\theta_x = \tau_{q, x}$, $\theta_y = \tau_{q, y}$, then $f = g(x^{\bar n}y^{-\bar m})$ for some $g\in K(z)$;
\item[3.] if $\theta_x = \si_x$, $\theta_y = \tau_{q, y}$, then $f \in K(x)$ if $m=0$, $f\in K(y)$ if $n=0$, and $f\in K$ if $mn\neq 0$.
\end{itemize}
%In particular, we have $c\in K$ for any choice of the pair $(\theta_x, \theta_y)$.
\end{proposition}
%\begin{proof}
%%Let $d = \deg_y(b)$ and write $b = \bar b/ u$, where $u\in K[x]$ and $\bar b = \sum_{i=0}^d b_iy^i$ with $b_i\in K[x]$ satisfying that $\gcd(b_0, b_1, \ldots, b_d)=1$ and $b_d = 1$. In the first case that $\theta_x = \si_x$ and $\theta_y = \si_y$, we have the relation
%%\[  \sum_{i=0}^d b_i(x+m) y^i  = c \sum_{i=0}^d b_i(x) (y+n)^i.\]
%%By comparing the leading coefficient in $y$, we get $c=1$. Then $\si_x^m(u) = u$ and $\si_x^m(\bar b) = \si_y^n(\bar b)$.
%%If $m=0$, then $\bar b \in K[x]$ since $n\neq 0$. So $b\in K(x)$. If $m\neq 0$, then $u\in K$ and $\bar b = \bar p(nx+my)$ for some
%%$\bar p \in K[z]$ by Proposition 7 in~\cite{AbramovPetkovsek2002a}. So $b = p(nx+my)$ with $p = \bar p/u\in K[z]$.
%%In the second case that $\theta_x = \si_x$, $\theta_y = \tau_{q, y}$, we have the relation
%%\[  \sum_{i=0}^d b_i(x+m) y^i  = c \sum_{i=0}^d q^{in}b_i(x) y^i.\]
%%By comparing the leading coefficient in $y$, we get $c=q^{-nd} \in K$. By Lemma 5.4 in~\cite{CCFFL2015} and the assumption that $y\nmid b$, we have $b\in K[y]$ if $m\neq 0$
%%and $b\in K[x]$ if $n\neq 0$.  Then $c =1$ and $b\in K$ if $mn\neq 0$. In the third case that $\theta_x = \tau_{q, x}$, $\theta_y = \tau_{q, y}$, we have the relation
%%\[  \sum_{i=0}^d b_i(q^m x) y^i  = c \sum_{i=0}^d q^{in}b_i(x) y^i.\]
%
%\EOP
%\end{proof}

We introduce a discrete analogue of the log-derivative pairs.
\begin{definition}\label{DEF:2WZp}
A WZ-pair $(f, g)$ with respect to $(\partial_x, \partial_y)$ is called a \emph{cyclic} pair if there exists a $(\theta_x, \theta_y)$-invariant $h\in K(x, y)$
such that
\[f = \frac{\theta_x^s - 1}{\theta_x-1} \bullet h \quad \text{and} \quad g = \frac{\theta_y^t - 1}{\theta_y-1} \bullet h,
\]
where $s, t\in \bZ$ are not all zero satisfying that
$\theta_x^s(h)= \theta_y^t(h)$.
\end{definition}
In the above definition, we may always assume that $s\geq 0$.
Note that for any $n\in \bZ$ we have
\[\frac{\theta_y^n - 1}{\theta_y-1} = \left\{
                                                                       \begin{array}{ll}
                                                                        \sum_{j=0}^{n-1}\theta_y^j, & \hbox{$n\geq 0$;} \\
                                                                         -\sum_{j=1}^{-n}\theta_y^{-j},& \hbox{$n<0$.}
                                                                       \end{array}
                                                                     \right.\]

\begin{example} Let~$a\in K(y)$ and $b\in K(x)$. Then both $(a, 0)$ and $(0, b)$ are cyclic  by taking $h =a, s=1, t=0$  and $h=b, s=0, t=1$, respectively.
Let $p = 2x + 3y$. Then the pair $(f, g)$ with
\[f = \frac{1}{p} + \frac{1}{\si_x(p)} + \frac{1}{\si_x^2(p)} \quad \text{and} \quad g = \frac{1}{p} + \frac{1}{\si_y(p)}\]
is a cyclic WZ-pair with respect to $(\Delta_x, \Delta_y)$.
\end{example}

Let~$V_0 = K(x)[y]$ and $V_m$ be
the set of all rational functions of the form $\sum_{i=1}^I {a_i}/{b_i^m}$,
where~$m\in \bZ_+, a_i, b_i, \in K(x)[y]$, $\deg_y(a_i) < \deg_y(b_i)$ and the~$b_i$'s are distinct
irreducible polynomials in the ring $K(x)[y]$. By definition, the set $V_m$ forms a subspace of
$K(x, y)$ as a vector spaces over~$K(x)$. By the irreducible partial fraction decomposition,
any $f\in K(x, y)$ can be uniquely decomposed
into $f = f_0 + f_1 + \cdots + f_n$ with~$f_i \in V_i$ and so~$K(x, y) = \bigoplus_{i=0}^\infty V_i$. The following lemma shows that
the space $V_m$ is invariant under certain shift operators.

\begin{lemma} \label{LM:inv}
Let~$f\in V_m$ and~$P\in K(x)[\theta_x, \theta_y]$. Then $P(f)\in V_m$.
\end{lemma}
\begin{proof}
Let~$f = \sum_{i=1}^I a_i/b_i^m$ and~$P= \sum_{i, j} p_{i, j} \theta_x^i \theta_y^j$.
For any $\theta=\theta_x^i\theta_y^j$ with~$i, j, k \in \bZ$, $\theta(b_i)$ is still
irreducible and~$\deg_y(\theta(a_i))< \deg_y(\theta(b_i))$. Then all of the simple fractions
${p_{i, j}\theta_x^i\theta_y^j(a_i)}/{\theta_x^i\theta_y^j(b_i)^n}$ appearing in $P(f)$ are
proper in~$y$ and have irreducible denominators. If some of denominators are the same,
we can simplify them by adding the numerators to get a simple fraction. After this simplification,
we see that $P(f)$ can be written in the same form as $f$, so it is in~$V_m$. \EOP
\end{proof}

\begin{lemma}\label{LM:constant}
Let~$p$ be a monic polynomial in $K(x)[y]$. If $\theta_x^{m}(p) = c\theta_y^n(p)$ for some $c\in K(x)$ and $m, n\in \bZ$ with $m, n$ being not both zero, then $c\in K$.
\end{lemma}
\begin{proof}
Write $p = \sum_{i=0}^d p_i y^i$ with $p_i\in K(x)$ and $p_d =1$. Then
\[\theta_x^{m}(p) =  \sum_{i=0}^d \theta_x^m(p_i) y^i  = c \sum_{i=0}^d p_i \theta_y^n(y^i)= c\theta_y^n(p).\]
Comparing the leading coefficients in $y$ yields $c=1$ if $\theta_y = \si_y$ and $c = q^{-nd}$ if $\theta_y = \tau_{q, y}$. Thus, $c\in K$ because $q\in K$.
\EOP
\end{proof}

\begin{lemma}\label{LM:cyclic}
Let~$f\in K(x, y)$ be a rational function of the form
\[ f = \frac{a_0}{b^m} + \frac{a_1}{\theta_x(b^m)} + \cdots + \frac{a_n}{\theta_x^n(b^m)},\]
where $m\in \bZ_+, n\in \bN, a_0, a_1, \ldots, a_n \in K(x)[y]$ with $a_n\neq 0$ and $b\in K(x)[y]$ are such that $\deg_y(a_i)<\deg_y(b)$ and $b$ is an
irreducible and monic polynomial in $K(x)[y]$ such that $\theta_x^i(b)$
and $\theta_x^j(b)$ are not $\theta_y$-equivalent for all $i, j \in \{0, 1, \ldots, n\}$ with $i\neq j$.
If $\theta_x(f)- f = \theta_y(g) - g$ for some $g\in K(x, y)$, then $(f, g)$ is cyclic.
\end{lemma}
\begin{proof}
By a direct calculation, we have
\[
\theta_x(f)- f  = \frac{\theta_x(a_n)}{\theta_x^{n+1}(b^m)} - \frac{a_0}{b^m} + \frac{\theta_x(a_0)-a_1}{\theta_x(b^m)} + \cdots + \frac{\theta_x(a_{n-1})-a_n}{\theta_x^n(b^m)}.
\]
If $\theta_x(f) - f = \theta_y(g) - g$ for some $g\in K(x, y)$, then all of the $\theta_y$-residues at distinct $\theta_y$-orbits of $\theta_x(f) - f$
are zero by residue criteria in Section~\ref{SECT:res}. Since $b^m, \theta_x(b^m), \ldots, \theta_x^{n}(b^m)$ are in distinct $\theta_y$-orbits, $\theta_x^{n+1}(b^m)$
must be $\theta_y$-equivalent to one of them. Otherwise, we get
\[
a_0 = 0, \quad \theta_x(a_0)-a_1 = 0, \quad \ldots, \quad \theta_x(a_{n-1})-a_n = 0,\quad  \text{and}\quad  \theta_x(a_n) = 0.
\]
Since $\theta_x$ is an automorphism on $K(x, y)$, we have $a_0 = a_1 = \cdots = a_n =0$, which contradicts the assumption that $a_n\neq 0$.
If $\theta_x^{n+1}(b^m)$ is $\theta_y$-equivalent to $\theta_x^i(b^m)$ for some $0<i\leq n$, so is
$\theta_x^{n+1-i}(b^m)$, which contradicts the assumption. Thus,  $\theta_x^{n+1}(b^m) = c\theta_y^t(b^m)$ for some $c\in K(x)\setminus\{0\}$ and $t\in \bZ$.
By Lemma~\ref{LM:constant}, we have $c\in K\setminus\{0\}$. A direct calculation leads to
\begin{align*}
  \theta_x(f)- f &  {=} \frac{\theta_x(a_{n})}{\theta_x^{n+1}(b^m)} {-} \frac{a_0}{b^m} + \sum_{i=1}^{n} \frac{\theta_x(a_{i-1})-a_{i}}{\theta_x^i(b^m)} {=} \frac{\theta_x(a_{n})}{c\theta_y^{t}(b^m)} {-} \frac{a_0}{b^m} + \sum_{i=1}^{n} \frac{\theta_x(a_{i-1})-a_{i}}{\theta_x^i(b^m)}\\
    &  {=} \frac{\theta_y^{-t}\theta_x(a_{n}/c)-a_0}{b^m}  + \sum_{i=1}^{n} \frac{\theta_x(a_{i-1})-a_{i}}{\theta_x^i(b^m)} + \theta_y(u) - u
\end{align*}
for some $u\in K(x, y)$ using the formula~\eqref{EQ:phi}. By the residue criteria,  we then get $a_0 = \theta_y^{-t}\theta_x(a_{n}/c), a_1 = \theta_x(a_0),  \ldots,$ and~$a_{n} = \theta_x(a_{n-1})$.
This implies that $\theta_x^{n+1}(a_0) = c \theta_y^t(a_0)$ and $a_i = \theta_x^i(a_0)$ for $i\in \{1, \ldots, n\}$. So $f = \frac{\theta_x^{n+1}-1}{\theta_x-1} \bullet h$
with $h= a_0/b^m$, which leads to
\[ \theta_x(f)- f = \theta_x^{n+1}(h)- h = \theta_y^t(h) - h = \theta_y(g) - g \quad \text{with}\quad  g = \frac{\theta_y^t-1}{\theta_y-1} \bullet h. \]
Thus, $(f, g)$ is a cyclic WZ-pair. \EOP

\end{proof}

%Any rational function $f\in K(x, y)$ can be decomposed into the form
%\[f = p + \sum_{i=1}^I \sum_{j=1}^{J_i} \sum_{k=0}^{K_{i, j}}\sum_{\ell=0}^{L_{i, j}} \frac{a_{i, j, k, \ell}}{\theta_y^k\theta_x^\ell(b_i^j)},\]
%where $p, a_{i, j, k, \ell}\in K(x)[y], b_i\in K[x, y]$ satisfying that $\deg_y(a_{i, j, k, \ell})<\deg_y(b_i)$ and $b_i$'s are irreducible polynomials in
%distinct $\langle \theta_x, \theta_y \rangle$-orbits.

The following theorem is a discrete analogue of Theorem~\ref{THM:cstructure}.

\begin{theorem}\label{THM:dstructure}
Let $f, g\in K(x, y)$ be such that $\partial_x(f)  = \partial_y(g)$. Then there exist rational functions $a, b_1, \ldots, b_n \in K(x, y)$ such that
\[f  = \partial_y(a) + \sum_{i =1}^n \frac{\theta_x^{s_i}-1}{\theta_x - 1}\bullet b_i \quad \text{and} \quad g  = \partial_x(a) + \sum_{i =1}^n \frac{\theta_y^{t_i}-1}{\theta_y - 1} \bullet b_i ,\]
where for each $i\in \{1, \ldots, n\}$ we have $\theta_x^{s_i}(b_i) = \theta_y^{t_i}(b_i)$ for some $s_i\in \bN$ and $t_i\in \bZ$ with $s_i, t_i$ not all zero.
\end{theorem}

\begin{proof}
By Abramov's reduction and its $q$-analogue, we can decompose $f$ as
\[f = \partial_y(a) + c + \sum_{j=1}^J f_j \quad \text{ with~$ f_j = \sum_{i=1}^I \sum_{\ell=0}^{L_{i, j}} \frac{a_{i, j, \ell}}{\theta_x^\ell(b_i^j)}$},\]
where $a\in K(x, y), c\in K(x)$,  and~$a_{i, j, \ell} b_i \in K(x)[y]$ such that
$c=0$ if $\theta_y = \si_y$, $\deg_y(a_{i, j, \ell})<\deg_y(b_i)$, and the $b_i$'s are irreducible and monic polynomials belonging to distinct $G$-orbits where $G=\langle \theta_x, \theta_y \rangle$.
Moreover, $\theta_x^{\ell_1}(b_i^j)$ and $\theta_x^{\ell_2}(b_i^j)$ are in distinct $\theta_y$-orbits if $\ell_1 \neq \ell_2$.
%A direct calculation leads to
%\[
%\theta_x(f_j)- f_j  = \sum_{i=1}^n \left(\left( \frac{\theta_x(a_{i, j, L_{i, j}})}{\theta_x^{L_{i, j}+1}(b_i^j)} - \frac{a_{i, j, 0}}{b_i^j} \right) + \sum_{\ell=1}^{L_{i, j}}\frac{\theta_x(a_{i, j, \ell-1})-
%a_{i, j, \ell}}{\theta_x^\ell (b_i^j)}\right) .
%\]
By applying Lemma~\ref{LM:inv} to the equation $\theta_x(f)-f = \theta_y(g)-g$, we get that $\theta_x(c) - c$ is $\theta_y$-summable and so is $\theta_x(f_j)-f_j$ for each multiplicity $j\in \{1, \ldots, J\}$.
By residue criteria for $\theta_y$-sumability and the assumption that the $b_i$'s are in distinct $\langle \theta_x, \theta_y \rangle$-orbits,
we have $\theta_x(c)-c = 0$ and for each $i\in \{1, \ldots, I\}$, the rational function $f_{i, j} := \sum_{\ell=0}^{L_{i, j}} {a_{i, j, \ell}}/{\theta_x^\ell(b_i^j)}$
is either equal to zero or there exists $g_{i, j}\in K(x, y)$ such that $\theta_x(f_{i, j})-f_{i, j} = \theta_y(g_{i, j})-g_{i, j}$. Then $(f_{i, j}, g_{i, j})$
is cyclic by Lemma~\ref{LM:cyclic} for every $i, j$ with $1\leq i \leq I$ and $1\leq j \leq J$. So the pair $(f, g)$ can be written as
\[(f, g) = (\partial_y(a), \partial_x(a)) + (c, 0) + \sum_{i=1}^I\sum_{j=1}^J (f_{i, j}, g_{i, j}).\]
This completes the proof. \EOP
\end{proof}

\begin{corollary}
The quotient space $\cP_{(\partial_x, \partial_y)}/\cE_{(\partial_x, \partial_y)}$ is spanned over $K$ by the set
\[ \{(f, g)+\cE_{(\partial_x, \partial_y)} \mid \text{$f, g\in K(x, y)$ such that $(f, g)$ is a cyclic pair}\}.\]
%\[ \left\{\left( \frac{\theta_x^s -1}{ \theta_x -1} \bullet b,\quad  \frac{\theta_y^t - 1}{\theta_y - 1}\bullet b \right) + \cE_{(\partial_x, \partial_y)} \mid \text{$b \in K(x, % % y)$ s.t. $\theta_x^s(b) =\theta_y^t(b)$ for some $s \in \bZ_+$, $t \in \bZ$ } \right\}.\]
\end{corollary}

%%\begin{remark}
%%
%%\end{remark}
\subsection{The mixed case} \label{SSECT:mixed}

In the mixed continuous-discrete setting, we consider the rational WZ-pairs with respect to $(\theta_x-1, D_y)$ with $\theta_x\in \{\si_x, \tau_{q, x}\}$.

\begin{lemma}\label{LM:spoly}
Let $p$ be an irreducible and monic polynomial in $K(x)[y]$.
Then for any nonzero $m\in \bZ$,  we have either $\gcd(p, \theta_x^m(p))=1$ or $p\in K[y]$.
\end{lemma}
\begin{proof}
Since $\theta_x$ is an automorphism on $K(x, y)$, $\theta_x^i(p)$ is irreducible in $K(x)[y]$ for any $i\in \bZ$.
If $\gcd(p, \theta_x^m(p)) \neq 1$, then $\theta_x^m(p)= c p$ for some $c\in K(x)$. Write $p = \sum_{i=0}^d p_i y^i$ with $p_i\in K(x)$ and $p_d= 1$.
Then $\theta_x^m(p)= c p$ implies that $\theta_x^m(p_i) = c p_i$ for all $i$ with $0\leq i \leq d$. Then $c=1$ and $p_i\in K$ for all $i$ with $0\leq i \leq d-1$.
So $p\in K[y]$. \EOP
\end{proof}

The structure of WZ-pairs in the mixed setting is as follows.

\begin{theorem}\label{THM:mstructure}
Let $f, g\in K(x, y)$ be such that $\theta_x(f) -f   = D_y(g)$. Then there exist $h\in K(x, y)$, $u\in K(y)$ and $v\in K(x)$ such that
\[f  = D_y(h) + u \quad \text{and} \quad g  = \theta_x(h) - h  + v.\]
\end{theorem}

\begin{proof}
By the Ostrogradsky--Hermite reduction, we decompose $f$ into the form
\[f = D_y(h) + \sum_{i=1}^I \sum_{j=0}^{J_i} \frac{a_{i, j}}{\theta_x^j(b_i)},\]
where $h\in K(x, y)$ and $a_{i, j}, b_i \in K(x)[y]$ with $a_{i, J_i}\neq 0$, $\deg_y(a_{i, j})<\deg_y(b_i)$ and
$b_i$ being irreducible and monic polynomials in $y$ over $K(x)$ such that the $b_i$'s are in distinct $\theta_x$-orbits.
By a direct calculation, we get
\[\theta_x(f)  - f = D_y(\theta_x(h)- h)+ \sum_{i=1}^I \left(\frac{\theta_x(a_{i, J_i})}{\theta_x^{J_i+1}(b_i)} - \frac{a_{i, 0}}{b_i} + \sum_{j=1}^{J_i} \frac{\theta_x(a_{i, j-1})-a_{i, j}}{\theta_x^j(b_i)}\right) .\]
For all $i, j$ with $1\leq i \leq I$ and $0\leq j \leq J_i + 1$, the $\theta_x^j(b_i)$'s are irreducible and monic polynomials in $y$ over $K(x)$.
We first show that for each $i\in \{1, \ldots, I\}$, we have $b_i\in K[y]$. Suppose that there exists $i_0\in \{1, \ldots, I\}$, $b_{i_0}\notin K[y]$. Then $\gcd(\theta_x^{m}(b_{i_0}), b_{i_0})=1$ for any nonzero $m\in \bZ$ by Lemma~\ref{LM:spoly}. Since $\theta_x(f)-f$ is $D_y$-integrable in $K(x, y)$, we have
$\theta_x(a_{i_0, J_{i_0}}) = 0$ by Lemma~\ref{LM:dres}. Then $a_{i_0, J_{i_0}}=0$, which contradicts the assumption that $a_{i, J_i}\neq 0$ for all $i$
with $1\leq i \leq I$. Since $b_i\in K[y]$, $f$ can be written as
\[f = D_y(h) + \sum_{i=1}^I \frac{a_i}{b_i}, \quad \text{where $a_i := \sum_{j=0}^{J_i} a_{i, j} $.}\]
Since $\theta_x(f)-f$ is $D_y$-integrable in $K(x, y)$ and since
\[ \theta_x(f)- f = D_y(\theta_x(h)-h) + \sum_{i=1}^I\frac{\theta_x(a_i)-a_i}{b_i},\]
we have  $\theta_x(a_i)-a_i = 0$ for each $i\in \{1, \ldots, I\}$ by Lemma~\ref{LM:dres}. This implies that $a_i\in K(y)$ and $f = D_y(h) + u$ with $u = \sum_{i=1}^I a_i/b_i\in K(y)$. Since $\theta_x(f)-f = D_y(g)$, we get $D_y(g - (\theta_x(h)-h))=0$. Then $g = \theta_x(h)-h + v$ for some $v\in K(x)$. \EOP
\end{proof}

\begin{corollary}
The quotient space $\cP_{(\theta_x-1, D_y)}/\cE_{(\theta_x-1, D_y)}$ is spanned over $K$ by the set
\[ \{(f, g)+\cE_{(\theta_x-1, D_y)} \mid \text{$f\in K(y)$ and $g\in K(x)$}\}.\]
%\[ \left\{\left( \frac{\theta_x^s -1}{ \theta_x -1} \bullet b,\quad  \frac{\theta_y^t - 1}{\theta_y - 1}\bullet b \right) + \cE_{(\partial_x, \partial_y)} \mid \text{$b \in K(x, % % y)$ s.t. $\theta_x^s(b) =\theta_y^t(b)$ for some $s \in \bZ_+$, $t \in \bZ$ } \right\}.\]
\end{corollary}

%\begin{lemma}\label{LM:torfree}
%Let $f\in K(x, y)$. Then $\theta_x^m(f)=f$ for some nonzero $m\in \bZ$ if and only if $f\in K(y)$.
%\end{lemma}
%\begin{proof}
%Write $f = p/q$ with $p, q\in K[x, y]$ and $\gcd(p, q)=1$. Then $\theta_x^m(f)=f$ if and only if $\theta_x^m(p)=p$ and $\theta_x^m(q)=q$.
%It suffices to show that for any $p\in K[x, y]$, $\theta_x^m(p)=p$ for some nonzero $m\in \bZ$ if and only $p\in K[y]$.
%Suppose that $\theta_x^m(p)=p$ for some nonzero $m\in \bZ$ but $\deg_x(p)>0$. Then $p$ as a polynomial in $x$ has at least one root in $\overline{K(y)}$, say $\alpha$. We consider the set $[\alpha]_p := \{\beta\in \overline{K(y)} \mid \text{$p(\beta)=0$ and $ \beta-\alpha \in \bZ$}\}$. With loss of generality, we may assume that $m>0$.
%\end{proof}

%\section{Selected applications of WZ-pairs}

\section{Conclusion} \label{SECT:conc}
We have explicitly described the structure of rational WZ-pairs in terms of special pairs.
With structure theorems, we can easily generate rational WZ-pairs, which solves Problem~\ref{PB:WZ}
in the rational case completely.
For the future research, the next direction is to solve
the problem in the cases of more general functions.
Using the terminology of Gessel in~\cite{Gessel1995}, a hypergeometric term $F(x, y)$ is said to be a \emph{WZ-function} if
there exists another hypergeometric term $G(x, y)$ such that $(F, G)$ is a WZ-pair. In the scheme of creative telescoping,
$(F, G)$ being a WZ-pair with respect to $(\partial_x, \partial_y)$ is equivalent to that
$\partial_x$ being a telescoper for $F$ with certificate $G$. Complete criteria for the
existence of telescopers for hypergeometric terms and their variants are known~\cite{Abramov2003, ChenHouMu2005, CCFFL2015}.
With the help of existence criteria for telescopers, one can show that $F(x, y)$ can be decomposed as the sum
$F = \partial_y(H_1) + H_2$ with $H_1, H_2$ being hypergeometric terms and $H_2$ is of proper form (see definition in~\cite{WilfZeilberger1992, Gessel1995})
if $F$ is a WZ-function.
So it is promising to apply the ideas in the study of the existence problem of telescopers to explore the structure of WZ-pairs.

\bigskip
\noindent {\bf Acknowledgment.}~
I would like to thank Prof.\ Victor J.W. Guo and Prof.\ Zhi-Wei Sun for many discussions on series for special
constants, (super)-congruences and their $q$-analogues that
can be proved using the WZ method. I am also very grateful to Ruyong Feng and Rong-Hua Wang for many constructive
comments on the earlier version of this paper. I also thank the anonymous reviewers for their constructive
and detailed comments.

\bibliographystyle{plain}
%\bibliography{mybib}

\end{document}